\documentclass{amsart}
\usepackage{amsfonts}

\newtheorem{theorem}{Theorem}
\theoremstyle{plain}
\newtheorem{acknowledgement}{Acknowledgement}

\newtheorem{corollary}{Corollary}

\newtheorem{lemma}{Lemma}

\newtheorem{remark}{Remark}

\numberwithin{equation}{section}
\input{tcilatex}

\begin{document}
\title[The equivalence among various iterative schemes]{The equivalence
among new multistep iteration, s-iteration and some other iterative schemes}
\author{Faik G\"{U}RSOY}
\address{Department of Mathematics, Yildiz Technical University, Davutpasa
Campus, Esenler, 34220 Istanbul, Turkey}
\email{faikgursoy02@hotmail.com;fgursoy@yildiz.edu.tr}
\urladdr{http://www.yarbis.yildiz.edu.tr/fgursoy}
\author{Vatan KARAKAYA}
\curraddr{Department of Mathematical Engineering, Yildiz Technical
University, Davutpasa Campus, Esenler, 34210 Istanbul}
\email{vkkaya@yildiz.edu.tr;vkkaya@yahoo.com}
\urladdr{http://www.yarbis.yildiz.edu.tr/vkkaya}
\author{B. E. RHOADES}
\address{Department of Mathematics, Indiana University, Bloomington, IN
47405-7106, USA}
\email{rhoades@indiana.edu}
\urladdr{http://www.math.indiana.edu/people/profile.phtml?id=rhoades}
\subjclass[2000]{Primary 47H10.}
\keywords{New multistep iteration, S-iteration, Equivalence of iterations,
Contractive-like mappings.}

\begin{abstract}
In this paper, we show that Picard, Krasnoselskij, Mann, Ishikawa, new two
step, Noor, multistep, new multistep, SP and S-iterative schemes are
equivalent for contractive-like mappings.
\end{abstract}

\maketitle

\section{Introduction and Preliminaries}

In the last four decades, attention of researchers has been focused on the
introduction and the convergences of various iteration procedures for
approximate fixed points of certain classes of self- nonlinear mappings,
e.g. see \cite{RS7, New, Noor, SP, Agarwall, Glowinski, Ishikawa, Thianwan,
Mann}.

The most celebrated fixed point iterative procedures are the Picard \cite%
{Picard}, Mann \cite{Mann}, and Ishikawa \cite{Ishikawa} iterative
procedures. Numerous convergence results have been proved through these
iterative procedures for approximating fixed points of different type
nonlinear mappings, e.g. see \cite{Ishikawa, Berinde1, Berinde,
Mann,Yuguang, Zhenyu}. But in some cases, some particular iteration
procedure may fail to converge for some class of nonlinear mappings. For
instance, (i) the Picard iteration procedure \cite{Picard} does not
convergence to the fixed point of nonexpansive mappings, (for more detail
see pp.8, Example 1.8 in \cite{Berinde2}),while the Ishikawa iteration \cite%
{Ishikawa} and Mann iteration \cite{Mann} converges. (ii) By providing a
counter example, Chidume and Mutangadura \cite{Chidume} showed that the Mann
iteration \cite{Mann} fails to converge for the class of Lipschitzian
pseudocontractive mappings while the Ishikawa iteration \cite{Ishikawa}
converges.

In the light of the above facts, a conjecture was put forwad in \cite{RS5,
RS7} as follows: While the Mann iteration \cite{Mann} converges to a fixed
point of a particular class of mappings, does the Ishikawa iteration \cite%
{Ishikawa} converges too? During the past 11 years, this conjecture was
proven affirmatively by many researchers and consequently a large literature
has developed around the theme of establishing the equivalence among
convergences of some well-known iterative schemes deal with various classes
of mappings. Some authors who have made contributions to the study of
equivalence among various iterative schemes are Rhoades and \c{S}oltuz \cite%
{RS1, RS2, RS3, RS4, RS5, RS6, RS7, RS8}, Berinde \cite{Berinde1}, \c{S}%
oltuz \cite{S1,S2}, Olaleru and Akewe \cite{Olaleru}, Chang et al \cite%
{Chang} and several of the references therein.

The main objective of this paper is attepmt to verify the above conjecture
for a new multistep iteration \cite{New} and some other well-known iterative
procedures in the literature.

As a background for our exposition, we now mention some contractive mappings
and iteration schemes.

In \cite{Zamfirescu} Zamfirescu established an important generalization of
the Banach fixed point theorem using the following contractive condition:
For a mapping $T:E\rightarrow E$, there exist real numbers $a,b,c$
satisfying $0<a<1$, $0<b,c<1/2$ such that, for each pair $x,y\in X$, at
least one of the following is true:%
\begin{equation}
\left\{ 
\begin{array}{c}
\text{(z}_{\text{1}}\text{) \ \ \ \ \ \ \ \ \ \ \ \ \ \ \ \ \ \ \ \ \ \ \ \ }%
\left\Vert Tx-Ty\right\Vert \leq a\left\Vert x-y\right\Vert \text{,} \\ 
\text{(z}_{\text{2}}\text{) \ \ \ \ }\left\Vert Tx-Ty\right\Vert \leq
b\left( \left\Vert x-Tx\right\Vert +\left\Vert y-Ty\right\Vert \right) \text{%
,} \\ 
\text{(z}_{\text{3}}\text{) \ \ \ \ }\left\Vert Tx-Ty\right\Vert \leq
c\left( \left\Vert x-Ty\right\Vert +\left\Vert y-Tx\right\Vert \right) \text{%
.}%
\end{array}%
\right.  \label{eqn1}
\end{equation}%
A mapping $T$ satisfying the contractive conditions (z$_{\text{1}}$), (z$_{%
\text{2}}$) and (z$_{\text{3}}$) in (1.1) is called a Zamfirescu mapping.

As shown in \cite{Berinde}, the contractive condition (1.1) leads to%
\begin{equation}
\left\{ 
\begin{array}{c}
\text{(b}_{\text{1}}\text{) \ \ \ \ }\left\Vert Tx-Ty\right\Vert \leq \delta
\left\Vert x-y\right\Vert +2\delta \left\Vert x-Tx\right\Vert \text{ if one
use (z}_{\text{2}}\text{),} \\ 
\text{and \ \ \ \ \ \ \ \ \ \ \ \ \ \ \ \ \ \ \ \ \ \ \ \ \ \ \ \ \ \ \ \ \
\ \ \ \ \ \ \ \ \ \ \ \ \ \ \ \ \ \ \ \ \ \ \ \ \ \ \ \ \ \ \ \ \ \ \ \ \ \
\ \ \ \ \ \ \ \ \ \ \ \ \ \ } \\ 
\text{(b}_{\text{2}}\text{) \ \ \ \ }\left\Vert Tx-Ty\right\Vert \leq \delta
\left\Vert x-y\right\Vert +2\delta \left\Vert x-Ty\right\Vert \text{ if one
use (z}_{\text{3}}\text{),}%
\end{array}%
\right.  \label{eqn2}
\end{equation}%
for all $x,y\in E$ where $\delta :=\max \left\{ a,\frac{b}{1-b},\frac{c}{1-c}%
\right\} $, $\delta \in \left[ 0,1\right) $, and it was shown that this
class of mappings is wider than the class of Zamfirescu mappings. Any
mapping satisfying condition (b$_{\text{1}}$) or (b$_{\text{2}}$) is called
a quasi-contractive mapping.

Extending the above definition, Osilike and Udomene \cite{Osilike}
considered mappings $T$ for which there exist real numbers $L\geq 0$ and $%
\delta \in \left[ 0,1\right) $ such that for all $x$, $y\in E$,%
\begin{equation}
\left\Vert Tx-Ty\right\Vert \leq \delta \left\Vert x-y\right\Vert
+L\left\Vert x-Tx\right\Vert .  \label{eqn3}
\end{equation}%
Imoru and Olantiwo \cite{Imoru} gave a more general definition: The mapping $%
T$\ is called a contractive-like mapping if there exists a constant $\delta
\in \left[ 0,1\right) $\ and a strictly increasing and continuous function $%
\varphi :\left[ 0,\infty \right) \rightarrow \left[ 0,\infty \right) $\ with 
$\varphi \left( 0\right) =0$,\ such that, for each $x,y\in E$,%
\begin{equation}
\left\Vert Tx-Ty\right\Vert \leq \delta \left\Vert x-y\right\Vert +\varphi
\left( \left\Vert x-Tx\right\Vert \right) .  \label{eqn4}
\end{equation}

\begin{remark}
\cite{New} A map satisfying (1.4) need not have a fixed point. However,
using (1.4), it is obvious that if T has a fixed point, then it is unique.
\end{remark}

Throughout the rest of this paper $%
\mathbb{N}
$ denotes the set of all nonnegative integers. Let $X$ be a Banach space and 
$E\subset X$\ be a nonempty closed, convex subset of $X$, and $T$ be a self
map on $E$. Define $F_{T}:=\left\{ p\in X:~p=Tp\right\} $ to be the set of
fixed points of $T$. Let $\left\{ \alpha _{n}\right\} _{n=0}^{\infty }$, $%
\left\{ \beta _{n}\right\} _{n=0}^{\infty }$,$\left\{ \gamma _{n}\right\}
_{n=0}^{\infty }$ and $\left\{ \beta _{n}^{i}\right\} _{n=0}^{\infty }$, $i=%
\overline{1,k-2}$, $k\geq 2$ be real sequences in $\left[ 0,1\right) $
satisfying certain conditions.

Rhoades and \c{S}oltuz \cite{RS7}, introduced a multistep iterative
algorithm by%
\begin{equation}
\left\{ 
\begin{array}{c}
x_{0}\in E\text{, \ \ \ \ \ \ \ \ \ \ \ \ \ \ \ \ \ \ \ \ \ \ \ \ \ \ \ \ \
\ \ \ \ \ \ \ \ \ \ \ \ \ \ \ \ } \\ 
x_{n+1}=\left( 1-\alpha _{n}\right) x_{n}+\alpha _{n}Ty_{n}^{1}\text{, \ \ \
\ \ \ \ \ \ \ \ \ \ \ \ \ \ } \\ 
y_{n}^{i}=\left( 1-\beta _{n}^{i}\right) x_{n}+\beta _{n}^{i}Ty_{n}^{i+1}%
\text{, \ \ \ \ \ \ \ \ \ \ \ } \\ 
y_{n}^{k-1}=\left( 1-\beta _{n}^{k-1}\right) x_{n}+\beta
_{n}^{k-1}Tx_{n},~n\in 
\mathbb{N}
\text{.}%
\end{array}%
\right.
\end{equation}%
The following multistep iteration was employed in \cite{New}%
\begin{equation}
\left\{ 
\begin{array}{c}
x_{0}\in E\text{, \ \ \ \ \ \ \ \ \ \ \ \ \ \ \ \ \ \ \ \ \ \ \ \ \ \ \ \ \
\ \ \ \ \ \ \ \ \ \ \ \ \ \ \ \ } \\ 
x_{n+1}=\left( 1-\alpha _{n}\right) y_{n}^{1}+\alpha _{n}Ty_{n}^{1}\text{, \
\ \ \ \ \ \ \ \ \ \ \ \ \ \ \ \ } \\ 
y_{n}^{i}=\left( 1-\beta _{n}^{i}\right) y_{n}^{i+1}+\beta
_{n}^{i}Ty_{n}^{i+1}\text{, \ \ \ \ \ \ \ \ } \\ 
y_{n}^{k-1}=\left( 1-\beta _{n}^{k-1}\right) x_{n}+\beta
_{n}^{k-1}Tx_{n},~n\in 
\mathbb{N}
\text{.}%
\end{array}%
\right.
\end{equation}%
By taking $k=3$ and $k=2$ in (1.5) we obtain the well-known Noor \cite{Noor}
and Ishikawa \cite{Ishikawa} iterative schemes, respectively. SP iteration 
\cite{SP} and a new two-step iteration \cite{Thianwan} processes are
obtained by taking $k=3$ and $k=2$ in (1.6), respectively. Both in (1.5) and
in (1.6), if we take $k=2$ with $\beta _{n}^{1}=0$ and $k=2$ with $\beta
_{n}^{1}\equiv 0$, $\alpha _{n}\equiv \lambda $ (const.), then we get the
iterative procedures introduced in \cite{Mann} and \cite{Krasnoselskij},
which are commonly known as the Mann and Krasnoselskij iterations,
respectively. The Krasnoselskij iteration reduces to the Picard iteration 
\cite{Picard} for $\lambda =1$.

A sequence $\left\{ x_{n}\right\} _{n=0}^{\infty }$ defined by%
\begin{equation}
\left\{ 
\begin{array}{c}
x_{0}\in E\text{, \ \ \ \ \ \ \ \ \ \ \ \ \ \ \ \ \ \ \ \ \ \ \ \ \ \ \ \ \
\ \ \ \ \ \ \ \ \ \ \ \ } \\ 
x_{n+1}=\left( 1-\alpha _{n}\right) Tx_{n}+\alpha _{n}Ty_{n}\text{, \ \ \ \
\ \ \ \ \ \ \ \ } \\ 
y_{n}=\left( 1-\beta _{n}\right) x_{n}+\beta _{n}Tx_{n}\text{,}~n\in 
\mathbb{N}
\text{ \ \ }%
\end{array}%
\right.  \label{eqn11}
\end{equation}%
is known as the S-iteration process \cite{Agarwal, Agarwall}.

The following lemma will be useful to prove the main results of this work
and is important by itself.

\begin{lemma}
\cite{Weng} Let $\left\{ a_{n}\right\} _{n=0}^{\infty }$ be a nonnegative
sequence which satisfies the following inequality%
\begin{equation}
a_{n+1}\leq \left( 1-\mu _{n}\right) a_{n}+\rho _{n}\text{,}  \label{eqn18}
\end{equation}%
where $\mu _{n}\in \left( 0,1\right) ,$ for all $n\geq n_{0}$, $%
\sum\limits_{n=0}^{\infty }\mu _{n}=\infty $, and $\rho _{n}=o\left( \mu
_{n}\right) $. Then $\lim_{n\rightarrow \infty }a_{n}=0$.
\end{lemma}

\section{Main Results}

\begin{theorem}
Let $T:E\rightarrow E$ be a mapping satisfying condition $\left( 1.4\right) $
with $F_{T}\neq \emptyset $. If $x_{0}=u_{0}\in E$ and $\alpha _{n}\geq A>0$,%
$\forall n\in 
\mathbb{N}
$, then the following are equivalent:

\begin{enumerate}
\item The Mann iteration \cite{Mann} converges to $p\in F_{T}$,

\item The new multistep iteration $\left( 1.6\right) $ converges to $p\in
F_{T}$.
\end{enumerate}
\end{theorem}

\begin{proof}
We first prove the implication $\left( 1\right) \Rightarrow \left( 2\right) $%
: Suppose that the Mann iteration \cite{Mann} converges to $p$. Using the
Mann iteration \cite{Mann}, (1.6), and (1.4) we have the following estimates:%
\begin{eqnarray}
\left\Vert u_{n+1}-x_{n+1}\right\Vert &=&\left\Vert \left( 1-\alpha
_{n}\right) \left( u_{n}-y_{n}^{1}\right) +\alpha _{n}\left(
Tu_{n}-Ty_{n}^{1}\right) \right\Vert  \notag \\
&\leq &\left( 1-\alpha _{n}\right) \left\Vert u_{n}-y_{n}^{1}\right\Vert
+\alpha _{n}\left\Vert Tu_{n}-Ty_{n}^{1}\right\Vert  \notag \\
&\leq &\left( 1-\alpha _{n}\right) \left\Vert u_{n}-y_{n}^{1}\right\Vert
+\alpha _{n}\left\{ \delta \left\Vert u_{n}-y_{n}^{1}\right\Vert +\varphi
\left( \left\Vert u_{n}-Tu_{n}\right\Vert \right) \right\}  \notag \\
&=&\left[ 1-\alpha _{n}\left( 1-\delta \right) \right] \left\Vert
u_{n}-y_{n}^{1}\right\Vert +\alpha _{n}\varphi \left( \left\Vert
u_{n}-Tu_{n}\right\Vert \right) \text{,}  \label{eqn19}
\end{eqnarray}%
\begin{eqnarray}
\left\Vert u_{n}-y_{n}^{1}\right\Vert &=&\left\Vert u_{n}-\left( 1-\beta
_{n}^{1}\right) y_{n}^{2}-\beta _{n}^{1}Ty_{n}^{2}\right\Vert  \notag \\
&=&\left\Vert u_{n}-\beta _{n}^{1}u_{n}+\beta _{n}^{1}u_{n}-\left( 1-\beta
_{n}^{1}\right) y_{n}^{2}-\beta _{n}^{1}Ty_{n}^{2}\right\Vert  \notag \\
&=&\left\Vert \left( 1-\beta _{n}^{1}\right) \left( u_{n}-y_{n}^{2}\right)
+\beta _{n}^{1}\left( u_{n}-Ty_{n}^{2}\right) \right\Vert  \notag \\
&\leq &\left( 1-\beta _{n}^{1}\right) \left\Vert u_{n}-y_{n}^{2}\right\Vert
+\beta _{n}^{1}\left\Vert u_{n}-Ty_{n}^{2}\right\Vert  \notag \\
&=&\left( 1-\beta _{n}^{1}\right) \left\Vert u_{n}-y_{n}^{2}\right\Vert
+\beta _{n}^{1}\left\Vert u_{n}-Tu_{n}+Tu_{n}-Ty_{n}^{2}\right\Vert  \notag
\\
&\leq &\left( 1-\beta _{n}^{1}\right) \left\Vert u_{n}-y_{n}^{2}\right\Vert
+\beta _{n}^{1}\left\Vert Tu_{n}-Ty_{n}^{2}\right\Vert +\beta
_{n}^{1}\left\Vert u_{n}-Tu_{n}\right\Vert  \notag \\
&\leq &\left( 1-\beta _{n}^{1}\right) \left\Vert u_{n}-y_{n}^{2}\right\Vert
+\beta _{n}^{1}\delta \left\Vert u_{n}-y_{n}^{2}\right\Vert +\beta
_{n}^{1}\varphi \left( \left\Vert u_{n}-Tu_{n}\right\Vert \right)  \notag \\
&&+\beta _{n}^{1}\left\Vert u_{n}-Tu_{n}\right\Vert  \notag \\
&=&\left[ 1-\beta _{n}^{1}\left( 1-\delta \right) \right] \left\Vert
u_{n}-y_{n}^{2}\right\Vert +\beta _{n}^{1}\left\{ \left\Vert
u_{n}-Tu_{n}\right\Vert +\varphi \left( \left\Vert u_{n}-Tu_{n}\right\Vert
\right) \right\} \text{,}  \label{eqn20}
\end{eqnarray}%
\begin{eqnarray}
\left\Vert u_{n}-y_{n}^{2}\right\Vert &=&\left\Vert \left( 1-\beta
_{n}^{2}\right) \left( u_{n}-y_{n}^{3}\right) +\beta _{n}^{2}\left(
u_{n}-Ty_{n}^{3}\right) \right\Vert  \notag \\
&\leq &\left( 1-\beta _{n}^{2}\right) \left\Vert u_{n}-y_{n}^{3}\right\Vert
+\beta _{n}^{2}\left\Vert u_{n}-Ty_{n}^{3}\right\Vert  \notag \\
&\leq &\left( 1-\beta _{n}^{2}\right) \left\Vert u_{n}-y_{n}^{3}\right\Vert
+\beta _{n}^{2}\left\Vert Tu_{n}-Ty_{n}^{3}\right\Vert +\beta
_{n}^{2}\left\Vert u_{n}-Tu_{n}\right\Vert  \notag \\
&\leq &\left( 1-\beta _{n}^{2}\right) \left\Vert u_{n}-y_{n}^{3}\right\Vert
+\beta _{n}^{2}\delta \left\Vert u_{n}-y_{n}^{3}\right\Vert +\beta
_{n}^{2}\varphi \left( \left\Vert u_{n}-Tu_{n}\right\Vert \right)  \notag \\
&&+\beta _{n}^{2}\left\Vert u_{n}-Tu_{n}\right\Vert  \notag \\
&=&\left[ 1-\beta _{n}^{2}\left( 1-\delta \right) \right] \left\Vert
u_{n}-y_{n}^{3}\right\Vert +\beta _{n}^{2}\left\{ \left\Vert
u_{n}-Tu_{n}\right\Vert +\varphi \left( \left\Vert u_{n}-Tu_{n}\right\Vert
\right) \right\} \text{,}  \label{eqn21}
\end{eqnarray}%
\begin{eqnarray}
\left\Vert u_{n}-y_{n}^{3}\right\Vert &=&\left\Vert \left( 1-\beta
_{n}^{3}\right) \left( u_{n}-y_{n}^{4}\right) +\beta _{n}^{3}\left(
u_{n}-Ty_{n}^{4}\right) \right\Vert  \notag \\
&\leq &\left( 1-\beta _{n}^{3}\right) \left\Vert u_{n}-y_{n}^{4}\right\Vert
+\beta _{n}^{3}\left\Vert u_{n}-Ty_{n}^{4}\right\Vert  \notag \\
&\leq &\left( 1-\beta _{n}^{3}\right) \left\Vert u_{n}-y_{n}^{4}\right\Vert
+\beta _{n}^{3}\left\Vert Tu_{n}-Ty_{n}^{4}\right\Vert +\beta
_{n}^{3}\left\Vert u_{n}-Tu_{n}\right\Vert  \notag \\
&\leq &\left( 1-\beta _{n}^{3}\right) \left\Vert u_{n}-y_{n}^{4}\right\Vert
+\beta _{n}^{3}\delta \left\Vert u_{n}-y_{n}^{4}\right\Vert +\beta
_{n}^{3}\varphi \left( \left\Vert u_{n}-Tu_{n}\right\Vert \right)  \notag \\
&&+\beta _{n}^{3}\left\Vert u_{n}-Tu_{n}\right\Vert  \notag \\
&=&\left[ 1-\beta _{n}^{3}\left( 1-\delta \right) \right] \left\Vert
u_{n}-y_{n}^{4}\right\Vert +\beta _{n}^{3}\left\{ \left\Vert
u_{n}-Tu_{n}\right\Vert +\varphi \left( \left\Vert u_{n}-Tu_{n}\right\Vert
\right) \right\} \text{.}  \label{eqn22}
\end{eqnarray}%
By combinig (2.1), (2.2), (2.3), and (2.4) we obtain%
\begin{eqnarray}
\left\Vert u_{n+1}-x_{n+1}\right\Vert &\leq &\left[ 1-\alpha _{n}\left(
1-\delta \right) \right] \left[ 1-\beta _{n}^{1}\left( 1-\delta \right) %
\right] \left[ 1-\beta _{n}^{2}\left( 1-\delta \right) \right]  \notag \\
&&\left[ 1-\beta _{n}^{3}\left( 1-\delta \right) \right] \left\Vert
u_{n}-y_{n}^{4}\right\Vert  \notag \\
&&+\left[ 1-\alpha _{n}\left( 1-\delta \right) \right] \left\{ \left[
1-\beta _{n}^{1}\left( 1-\delta \right) \right] \left[ 1-\beta
_{n}^{2}\left( 1-\delta \right) \right] \beta _{n}^{3}\right.  \notag \\
&&\left. +\left[ 1-\beta _{n}^{1}\left( 1-\delta \right) \right] \beta
_{n}^{2}+\beta _{n}^{1}\right\} \left\{ \left\Vert u_{n}-Tu_{n}\right\Vert
+\varphi \left( \left\Vert u_{n}-Tu_{n}\right\Vert \right) \right\}  \notag
\\
&&+\alpha _{n}\varphi \left( \left\Vert u_{n}-Tu_{n}\right\Vert \right)
\label{eqn23}
\end{eqnarray}%
Continuing the above process we have%
\begin{eqnarray}
\left\Vert u_{n+1}-x_{n+1}\right\Vert &\leq &\left[ 1-\alpha _{n}\left(
1-\delta \right) \right] \left[ 1-\beta _{n}^{1}\left( 1-\delta \right) %
\right] \cdots \left[ 1-\beta _{n}^{k-2}\left( 1-\delta \right) \right]
\left\Vert u_{n}-y_{n}^{k-1}\right\Vert  \notag \\
&&+\left[ 1-\alpha _{n}\left( 1-\delta \right) \right] \left\{ \left[
1-\beta _{n}^{1}\left( 1-\delta \right) \right] \cdots \left[ 1-\beta
_{n}^{k-3}\left( 1-\delta \right) \right] \beta _{n}^{k-2}\right.  \notag \\
&&\left. +\cdots +\left[ 1-\beta _{n}^{1}\left( 1-\delta \right) \right]
\beta _{n}^{2}+\beta _{n}^{1}\right\} \left\{ \left\Vert
u_{n}-Tu_{n}\right\Vert +\varphi \left( \left\Vert u_{n}-Tu_{n}\right\Vert
\right) \right\}  \notag \\
&&+\alpha _{n}\varphi \left( \left\Vert u_{n}-Tu_{n}\right\Vert \right) 
\text{.}  \label{eqn24}
\end{eqnarray}%
Again using (1.6), and (1.4) we get%
\begin{eqnarray}
\left\Vert u_{n}-y_{n}^{k-1}\right\Vert &=&\left\Vert \left( 1-\beta
_{n}^{k-1}\right) \left( u_{n}-x_{n}\right) +\beta _{n}^{k-1}\left(
u_{n}-Tx_{n}\right) \right\Vert  \notag \\
&\leq &\left( 1-\beta _{n}^{k-1}\right) \left\Vert u_{n}-x_{n}\right\Vert
+\beta _{n}^{k-1}\left\Vert u_{n}-Tx_{n}\right\Vert  \notag \\
&\leq &\left( 1-\beta _{n}^{k-1}\right) \left\Vert u_{n}-x_{n}\right\Vert
+\beta _{n}^{k-1}\left\Vert Tu_{n}-Tx_{n}\right\Vert +\beta
_{n}^{k-1}\left\Vert u_{n}-Tu_{n}\right\Vert  \notag \\
&\leq &\left[ 1-\beta _{n}^{k-1}\left( 1-\delta \right) \right] \left\Vert
u_{n}-x_{n}\right\Vert +\beta _{n}^{k-1}\left\{ \left\Vert
u_{n}-Tu_{n}\right\Vert +\varphi \left( \left\Vert u_{n}-Tu_{n}\right\Vert
\right) \right\} \text{.}  \label{eqn25}
\end{eqnarray}%
Since $\delta \in \left[ 0,1\right) $ and $\left\{ \alpha _{n}\right\}
_{n=0}^{\infty }$,$\left\{ \beta _{n}^{i}\right\} _{n=0}^{\infty }\subset %
\left[ 0,1\right) $ for $i=\overline{1,k-1}$, we have%
\begin{equation}
\left[ 1-\alpha _{n}\left( 1-\delta \right) \right] \left[ 1-\beta
_{n}^{1}\left( 1-\delta \right) \right] \cdots \left[ 1-\beta
_{n}^{k-1}\left( 1-\delta \right) \right] \leq \left[ 1-\alpha _{n}\left(
1-\delta \right) \right] \text{.}  \label{eqn26}
\end{equation}%
Using inequality (2.8) and the assumption $\alpha _{n}\geq A>0$,$\forall
n\in 
\mathbb{N}
$ in the resultant inequality obtained by substituting (2.7) in (2.6) we get%
\begin{eqnarray}
\left\Vert u_{n+1}-x_{n+1}\right\Vert &\leq &\left[ 1-A\left( 1-\delta
\right) \right] \left\Vert u_{n}-x_{n}\right\Vert  \notag \\
&&+\left[ 1-A\left( 1-\delta \right) \right] \left\{ \left[ 1-\beta
_{n}^{1}\left( 1-\delta \right) \right] \cdots \left[ 1-\beta
_{n}^{k-2}\left( 1-\delta \right) \right] \beta _{n}^{k-1}\right.  \notag \\
&&\left. +\cdots +\left[ 1-\beta _{n}^{1}\left( 1-\delta \right) \right]
\beta _{n}^{2}+\beta _{n}^{1}\right\} \left\{ \left\Vert
u_{n}-Tu_{n}\right\Vert +\varphi \left( \left\Vert u_{n}-Tu_{n}\right\Vert
\right) \right\}  \notag \\
&&+\alpha _{n}\varphi \left( \left\Vert u_{n}-Tu_{n}\right\Vert \right) 
\text{.}  \label{eqn27}
\end{eqnarray}%
Define%
\begin{eqnarray*}
a_{n} &:&=\left\Vert u_{n}-x_{n}\right\Vert \text{,} \\
\mu _{n} &:&=A\left( 1-\delta \right) \in \left( 0,1\right) \text{,} \\
\rho _{n} &:&=\left[ 1-A\left( 1-\delta \right) \right] \left\{ \left[
1-\beta _{n}^{1}\left( 1-\delta \right) \right] \cdots \left[ 1-\beta
_{n}^{k-2}\left( 1-\delta \right) \right] \beta _{n}^{k-1}\right. \\
&&\left. +\cdots +\left[ 1-\beta _{n}^{1}\left( 1-\delta \right) \right]
\beta _{n}^{2}+\beta _{n}^{1}\right\} \left\{ \left\Vert
u_{n}-Tu_{n}\right\Vert +\varphi \left( \left\Vert u_{n}-Tu_{n}\right\Vert
\right) \right\} \\
&&+\alpha _{n}\varphi \left( \left\Vert u_{n}-Tu_{n}\right\Vert \right) 
\text{.}
\end{eqnarray*}%
Since $\lim_{n\rightarrow \infty }\left\Vert u_{n}-p\right\Vert =0$ and $%
Tp=p\in F_{T}$, it follows from (1.4) that%
\begin{eqnarray}
0 &\leq &\left\Vert u_{n}-Tu_{n}\right\Vert  \notag \\
&\leq &\left\Vert u_{n}-p\right\Vert +\left\Vert Tp-Tu_{n}\right\Vert  \notag
\\
&\leq &\left\Vert u_{n}-p\right\Vert +\delta \left\Vert p-u_{n}\right\Vert
+\varphi \left( \left\Vert p-Tp\right\Vert \right)  \notag \\
&=&\left( 1+\delta \right) \left\Vert u_{n}-p\right\Vert \rightarrow 0\text{
as }n\rightarrow \infty \text{,}  \label{eqn28}
\end{eqnarray}%
which implies $\lim_{n\rightarrow \infty }\left\Vert u_{n}-Tu_{n}\right\Vert
=0$; namely $\rho _{n}=o\left( \mu _{n}\right) $. Hence an application of
Lemma 1 to (2.10) yields $\lim_{n\rightarrow \infty }\left\Vert
u_{n}-x_{n}\right\Vert =0$. Since $u_{n}\rightarrow p$ as $n\rightarrow
\infty $ by assumption, we derive%
\begin{equation}
\left\Vert x_{n}-p\right\Vert \leq \left\Vert x_{n}-u_{n}\right\Vert
+\left\Vert u_{n}-p\right\Vert  \label{eqn29}
\end{equation}%
and this implies that $\lim_{n\rightarrow \infty }x_{n}=p$.

$\left( 2\right) \Rightarrow \left( 1\right) :$ Assume that $%
x_{n}\rightarrow p$ as $n\rightarrow \infty $. Using the Mann iteration \cite%
{Mann}, (1.6), and (1.4), we have the following estimates:%
\begin{eqnarray}
\left\Vert x_{n+1}-u_{n+1}\right\Vert &=&\left\Vert \left( 1-\alpha
_{n}\right) \left( y_{n}^{1}-u_{n}\right) +\alpha _{n}\left(
Ty_{n}^{1}-Tu_{n}\right) \right\Vert  \notag \\
&\leq &\left( 1-\alpha _{n}\right) \left\Vert y_{n}^{1}-u_{n}\right\Vert
+\alpha _{n}\left\Vert Ty_{n}^{1}-Tu_{n}\right\Vert  \notag \\
&\leq &\left( 1-\alpha _{n}\right) \left\Vert y_{n}^{1}-u_{n}\right\Vert
+\alpha _{n}\left\{ \delta \left\Vert y_{n}^{1}-u_{n}\right\Vert +\varphi
\left( \left\Vert y_{n}^{1}-Ty_{n}^{1}\right\Vert \right) \right\}  \notag \\
&=&\left[ 1-\alpha _{n}\left( 1-\delta \right) \right] \left\Vert
y_{n}^{1}-u_{n}\right\Vert +\alpha _{n}\varphi \left( \left\Vert
y_{n}^{1}-Ty_{n}^{1}\right\Vert \right) \text{,}  \label{eqn30}
\end{eqnarray}%
\begin{eqnarray}
\left\Vert y_{n}^{1}-u_{n}\right\Vert &=&\left\Vert \left( 1-\beta
_{n}^{1}\right) y_{n}^{2}+\beta _{n}^{1}Ty_{n}^{2}-u_{n}\right\Vert  \notag
\\
&=&\left\Vert \left( 1-\beta _{n}^{1}\right) y_{n}^{2}+\beta
_{n}^{1}Ty_{n}^{2}-u_{n}\left( 1-\beta _{n}^{1}+\beta _{n}^{1}\right)
\right\Vert  \notag \\
&=&\left\Vert \left( 1-\beta _{n}^{1}\right) \left( y_{n}^{2}-u_{n}\right)
+\beta _{n}^{1}\left( Ty_{n}^{2}-u_{n}\right) \right\Vert  \notag \\
&\leq &\left( 1-\beta _{n}^{1}\right) \left\Vert y_{n}^{2}-u_{n}\right\Vert
+\beta _{n}^{1}\left\Vert Ty_{n}^{2}-u_{n}\right\Vert  \notag \\
&\leq &\left( 1-\beta _{n}^{1}\right) \left\Vert y_{n}^{2}-u_{n}\right\Vert
+\beta _{n}^{1}\left\Vert Ty_{n}^{2}-y_{n}^{2}+y_{n}^{2}-u_{n}\right\Vert 
\notag \\
&\leq &\left( 1-\beta _{n}^{1}\right) \left\Vert y_{n}^{2}-u_{n}\right\Vert
+\beta _{n}^{1}\left\Vert y_{n}^{2}-u_{n}\right\Vert +\beta
_{n}^{1}\left\Vert Ty_{n}^{2}-y_{n}^{2}\right\Vert  \notag \\
&=&\left\Vert y_{n}^{2}-u_{n}\right\Vert +\beta _{n}^{1}\left\Vert
Ty_{n}^{2}-y_{n}^{2}\right\Vert ,  \label{eqn31}
\end{eqnarray}%
\begin{eqnarray}
\left\Vert y_{n}^{2}-u_{n}\right\Vert &=&\left\Vert \left( 1-\beta
_{n}^{2}\right) y_{n}^{3}+\beta _{n}^{2}Ty_{n}^{3}-u_{n}\right\Vert  \notag
\\
&=&\left\Vert \left( 1-\beta _{n}^{2}\right) \left( y_{n}^{3}-u_{n}\right)
+\beta _{n}^{2}\left( Ty_{n}^{3}-u_{n}\right) \right\Vert  \notag \\
&\leq &\left( 1-\beta _{n}^{2}\right) \left\Vert y_{n}^{3}-u_{n}\right\Vert
+\beta _{n}^{2}\left\Vert Ty_{n}^{3}-u_{n}\right\Vert  \notag \\
&\leq &\left( 1-\beta _{n}^{2}\right) \left\Vert y_{n}^{3}-u_{n}\right\Vert
+\beta _{n}^{2}\left\Vert y_{n}^{3}-u_{n}\right\Vert +\beta
_{n}^{2}\left\Vert Ty_{n}^{3}-y_{n}^{3}\right\Vert  \notag \\
&=&\left\Vert y_{n}^{3}-u_{n}\right\Vert +\beta _{n}^{2}\left\Vert
Ty_{n}^{3}-y_{n}^{3}\right\Vert \text{.}  \label{eqn32}
\end{eqnarray}%
By combining (2.12), (2.13), and (2.14) we obtain%
\begin{eqnarray}
\left\Vert x_{n+1}-u_{n+1}\right\Vert &\leq &\left[ 1-\alpha _{n}\left(
1-\delta \right) \right] \left\Vert y_{n}^{3}-u_{n}\right\Vert +\left[
1-\alpha _{n}\left( 1-\delta \right) \right] \beta _{n}^{2}\left\Vert
Ty_{n}^{3}-y_{n}^{3}\right\Vert  \notag \\
&&+\left[ 1-\alpha _{n}\left( 1-\delta \right) \right] \beta
_{n}^{1}\left\Vert Ty_{n}^{2}-y_{n}^{2}\right\Vert +\alpha _{n}\varphi
\left( \left\Vert y_{n}^{1}-Ty_{n}^{1}\right\Vert \right)  \label{eqn33}
\end{eqnarray}%
In a similar way, we have 
\begin{eqnarray}
\left\Vert x_{n+1}-u_{n+1}\right\Vert &\leq &\left[ 1-\alpha _{n}\left(
1-\delta \right) \right] \left\Vert y_{n}^{k-1}-u_{n}\right\Vert  \notag \\
&&+\left[ 1-\alpha _{n}\left( 1-\delta \right) \right] \beta
_{n}^{k-2}\left\Vert Ty_{n}^{k-1}-y_{n}^{k-1}\right\Vert  \notag \\
&&+\cdots +\left[ 1-\alpha _{n}\left( 1-\delta \right) \right] \beta
_{n}^{1}\left\Vert Ty_{n}^{2}-y_{n}^{2}\right\Vert +\alpha _{n}\varphi
\left( \left\Vert y_{n}^{1}-Ty_{n}^{1}\right\Vert \right)  \label{eqn34}
\end{eqnarray}%
Using now (1.6) we have%
\begin{eqnarray}
\left\Vert y_{n}^{k-1}-u_{n}\right\Vert &=&\left\Vert \left( 1-\beta
_{n}^{k-1}\right) x_{n}+\beta _{n}^{k-1}Tx_{n}-u_{n}\right\Vert  \notag \\
&\leq &\left( 1-\beta _{n}^{k-1}\right) \left\Vert x_{n}-u_{n}\right\Vert
+\beta _{n}^{k-1}\left\Vert Tx_{n}-u_{n}\right\Vert  \notag \\
&\leq &\left( 1-\beta _{n}^{k-1}\right) \left\Vert x_{n}-u_{n}\right\Vert
+\beta _{n}^{k-1}\left\Vert x_{n}-u_{n}\right\Vert +\beta
_{n}^{k-1}\left\Vert Tx_{n}-x_{n}\right\Vert  \notag \\
&\leq &\left\Vert x_{n}-u_{n}\right\Vert +\beta _{n}^{k-1}\left\Vert
Tx_{n}-x_{n}\right\Vert \text{.}  \label{eqn35}
\end{eqnarray}%
Substituting (2.17) in (2.16) and utilizing the assumption $\alpha _{n}\geq
A>0$,$\forall n\in 
\mathbb{N}
$ we get%
\begin{eqnarray}
\left\Vert x_{n+1}-u_{n+1}\right\Vert &\leq &\left[ 1-A\left( 1-\delta
\right) \right] \left\Vert x_{n}-u_{n}\right\Vert  \notag \\
&&+\left[ 1-A\left( 1-\delta \right) \right] \left\{ \beta
_{n}^{k-1}\left\Vert Tx_{n}-x_{n}\right\Vert +\beta _{n}^{k-2}\left\Vert
Ty_{n}^{k-1}-y_{n}^{k-1}\right\Vert \right.  \notag \\
&&\left. +\cdots +\beta _{n}^{1}\left\Vert Ty_{n}^{2}-y_{n}^{2}\right\Vert
\right\} +\alpha _{n}\varphi \left( \left\Vert
y_{n}^{1}-Ty_{n}^{1}\right\Vert \right) \text{.}  \label{eqn36}
\end{eqnarray}%
Now define%
\begin{eqnarray*}
a_{n} &:&=\left\Vert u_{n}-x_{n}\right\Vert \text{,} \\
\mu _{n} &:&=A\left( 1-\delta \right) \in \left( 0,1\right) \text{,} \\
\rho _{n} &:&=\left[ 1-A\left( 1-\delta \right) \right] \left\{ \beta
_{n}^{k-1}\left\Vert Tx_{n}-x_{n}\right\Vert +\beta _{n}^{k-2}\left\Vert
Ty_{n}^{k-1}-y_{n}^{k-1}\right\Vert \right. \\
&&\left. +\cdots +\beta _{n}^{1}\left\Vert Ty_{n}^{2}-y_{n}^{2}\right\Vert
\right\} +\alpha _{n}\varphi \left( \left\Vert
y_{n}^{1}-Ty_{n}^{1}\right\Vert \right) \text{.}
\end{eqnarray*}%
Since $\lim_{n\rightarrow \infty }\left\Vert x_{n}-p\right\Vert =0$ and $%
Tp=p\in F_{T}$, it follows from (1.4) that%
\begin{eqnarray}
0 &\leq &\left\Vert x_{n}-Tx_{n}\right\Vert  \notag \\
&\leq &\left\Vert x_{n}-p\right\Vert +\left\Vert Tp-Tx_{n}\right\Vert  \notag
\\
&\leq &\left\Vert x_{n}-p\right\Vert +\delta \left\Vert p-x_{n}\right\Vert
+\varphi \left( \left\Vert p-Tp\right\Vert \right)  \notag \\
&=&\left( 1+\delta \right) \left\Vert x_{n}-p\right\Vert \rightarrow 0\text{
as }n\rightarrow \infty \text{.}  \label{eqn37}
\end{eqnarray}

Utilizing (1.4), (1.6), and the condition $\left\{ \beta _{n}^{i}\right\}
_{n=0}^{\infty }\subset \left[ 0,1\right) $, $i=\overline{1,k-1}$, we have%
\begin{eqnarray}
0 &\leq &\left\Vert y_{n}^{1}-Ty_{n}^{1}\right\Vert =\left\Vert
y_{n}^{1}-p+p-Ty_{n}^{1}\right\Vert  \notag \\
&\leq &\left\Vert y_{n}^{1}-p\right\Vert +\left\Vert Tp-Ty_{n}^{1}\right\Vert
\notag \\
&\leq &\left\Vert y_{n}^{1}-p\right\Vert +\delta \left\Vert
p-y_{n}^{1}\right\Vert +\varphi \left( \left\Vert p-Tp\right\Vert \right) 
\notag \\
&=&\left( 1+\delta \right) \left\Vert y_{n}^{1}-p\right\Vert  \notag \\
&=&\left( 1+\delta \right) \left\Vert \left( 1-\beta _{n}^{1}\right)
y_{n}^{2}+\beta _{n}^{1}Ty_{n}^{2}-p\left( 1-\beta _{n}^{1}+\beta
_{n}^{1}\right) \right\Vert  \notag \\
&\leq &\left( 1+\delta \right) \left\{ \left( 1-\beta _{n}^{1}\right)
\left\Vert y_{n}^{2}-p\right\Vert +\beta _{n}^{1}\left\Vert
Ty_{n}^{2}-Tp\right\Vert \right\}  \notag \\
&\leq &\left( 1+\delta \right) \left\{ \left( 1-\beta _{n}^{1}\right)
\left\Vert y_{n}^{2}-p\right\Vert +\beta _{n}^{1}\delta \left\Vert
y_{n}^{2}-p\right\Vert \right\}  \notag \\
&=&\left( 1+\delta \right) \left[ 1-\beta _{n}^{1}\left( 1-\delta \right) %
\right] \left\Vert y_{n}^{2}-p\right\Vert  \notag \\
&=&\left( 1+\delta \right) \left[ 1-\beta _{n}^{1}\left( 1-\delta \right) %
\right] \left\Vert \left( 1-\beta _{n}^{2}\right) y_{n}^{3}+\beta
_{n}^{2}Ty_{n}^{3}-p\left( 1-\beta _{n}^{2}+\beta _{n}^{2}\right) \right\Vert
\notag \\
&\leq &\left( 1+\delta \right) \left[ 1-\beta _{n}^{1}\left( 1-\delta
\right) \right] \left\{ \left( 1-\beta _{n}^{2}\right) \left\Vert
y_{n}^{3}-p\right\Vert +\beta _{n}^{2}\left\Vert Ty_{n}^{3}-Tp\right\Vert
\right\}  \notag \\
&\leq &\left( 1+\delta \right) \left[ 1-\beta _{n}^{1}\left( 1-\delta
\right) \right] \left[ 1-\beta _{n}^{2}\left( 1-\delta \right) \right]
\left\Vert y_{n}^{3}-p\right\Vert  \notag \\
&&\cdots  \notag \\
&\leq &\left( 1+\delta \right) \left[ 1-\beta _{n}^{1}\left( 1-\delta
\right) \right] \cdots \left[ 1-\beta _{n}^{k-2}\left( 1-\delta \right) %
\right] \left\Vert y_{n}^{k-1}-p\right\Vert  \notag \\
&\leq &\left( 1+\delta \right) \left[ 1-\beta _{n}^{1}\left( 1-\delta
\right) \right] \cdots \left[ 1-\beta _{n}^{k-1}\left( 1-\delta \right) %
\right] \left\Vert x_{n}-p\right\Vert  \notag \\
&\leq &\left( 1+\delta \right) \left\Vert x_{n}-p\right\Vert \rightarrow 0%
\text{ as }n\rightarrow \infty \text{.}  \label{eqn38}
\end{eqnarray}%
It is easy to see from (2.20) that this result is also valid for $\left\Vert
Ty_{n}^{2}-y_{n}^{2}\right\Vert ,\ldots ,\left\Vert
Ty_{n}^{k-1}-y_{n}^{k-1}\right\Vert $.

Since $\varphi $ is continuous, we have 
\begin{eqnarray}
\lim_{n\rightarrow \infty }\left\Vert x_{n}-Tx_{n}\right\Vert
&=&\lim_{n\rightarrow \infty }\varphi \left( \left\Vert
y_{n}^{1}-Ty_{n}^{1}\right\Vert \right)  \notag \\
&=&\lim_{n\rightarrow \infty }\left\Vert y_{n}^{2}-Ty_{n}^{2}\right\Vert
=\cdots =\lim_{n\rightarrow \infty }\left\Vert
y_{n}^{k-1}-Ty_{n}^{k-1}\right\Vert =0\text{,}  \label{eqn39}
\end{eqnarray}%
that is $\rho _{n}=o\left( \mu _{n}\right) $. Hence an application of Lemma
1 to (2.18) lead to $\lim_{n\rightarrow \infty }\left\Vert
x_{n}-u_{n}\right\Vert =0$. Since $x_{n}\rightarrow p$ as $n\rightarrow
\infty $ by assumption, we derive%
\begin{equation}
\left\Vert u_{n}-p\right\Vert \leq \left\Vert u_{n}-x_{n}\right\Vert
+\left\Vert x_{n}-p\right\Vert  \label{eqn40}
\end{equation}%
and this implies that $\lim_{n\rightarrow \infty }u_{n}=p$.
\end{proof}

\begin{theorem}
Let $T:E\rightarrow E$ be a mapping satisfying condition $\left( 1.4\right) $
with $F_{T}\neq \emptyset $. If $x_{0}=u_{0}\in E$ and $\alpha _{n}\geq A>0$,%
$\forall n\in 
\mathbb{N}
$, then the following are equivalent:

\begin{enumerate}
\item The Mann iteration \cite{Mann} converges to $p\in F_{T}$,

\item The S-iteration $\left( 1.7\right) $ converges to $p\in F_{T}$.
\end{enumerate}
\end{theorem}

\begin{proof}
To prove the implication $\left( 1\right) \Rightarrow \left( 2\right) $,
suppose that the Mann iteration \cite{Mann} converges to $p$. Using (1.4),
the Mann iteration \cite{Mann}, and (1.7) we have the following estimates:%
\begin{eqnarray}
\left\Vert u_{n+1}-x_{n+1}\right\Vert &=&\left\Vert \left( 1-\alpha
_{n}\right) \left( u_{n}-Tx_{n}\right) +\alpha _{n}\left(
Tu_{n}-Ty_{n}\right) \right\Vert  \notag \\
&\leq &\left( 1-\alpha _{n}\right) \left\Vert u_{n}-Tx_{n}\right\Vert
+\alpha _{n}\left\Vert Tu_{n}-Ty_{n}\right\Vert  \notag \\
&\leq &\left( 1-\alpha _{n}\right) \left\Vert u_{n}-Tx_{n}\right\Vert
+\alpha _{n}\delta \left\Vert u_{n}-y_{n}\right\Vert +\alpha _{n}\varphi
\left( \left\Vert u_{n}-Tu_{n}\right\Vert \right) \text{,}  \label{eqn41}
\end{eqnarray}%
\begin{eqnarray}
\left\Vert u_{n}-y_{n}\right\Vert &=&\left\Vert u_{n}-\left( 1-\beta
_{n}\right) x_{n}-\beta _{n}Tx_{n}\right\Vert  \notag \\
&=&\left\Vert u_{n}-\beta _{n}u_{n}+\beta _{n}u_{n}-\left( 1-\beta
_{n}\right) x_{n}-\beta _{n}Tx_{n}\right\Vert  \notag \\
&\leq &\left( 1-\beta _{n}\right) \left\Vert u_{n}-x_{n}\right\Vert +\beta
_{n}\left\Vert u_{n}-Tx_{n}\right\Vert \text{,}  \label{eqn42}
\end{eqnarray}%
\begin{eqnarray}
\left\Vert u_{n}-Tx_{n}\right\Vert &=&\left\Vert
u_{n}-Tu_{n}+Tu_{n}-Tx_{n}\right\Vert  \notag \\
&\leq &\left\Vert u_{n}-Tu_{n}\right\Vert +\left\Vert
Tu_{n}-Tx_{n}\right\Vert  \notag \\
&\leq &\left\Vert u_{n}-Tu_{n}\right\Vert +\delta \left\Vert
u_{n}-x_{n}\right\Vert +\varphi \left( \left\Vert u_{n}-Tu_{n}\right\Vert
\right) \text{.}  \label{eqn43}
\end{eqnarray}%
By combining (2.23),(2.24), and (2.25) we obtain%
\begin{eqnarray}
\left\Vert u_{n+1}-x_{n+1}\right\Vert &\leq &\left\{ \left( 1-\alpha
_{n}\right) \delta +\alpha _{n}\delta \left[ 1-\beta _{n}\left( 1-\delta
\right) \right] \right\} \left\Vert u_{n}-x_{n}\right\Vert  \notag \\
&&+\left[ 1-\alpha _{n}+\alpha _{n}\beta _{n}\delta \right] \left\Vert
u_{n}-Tu_{n}\right\Vert  \notag \\
&&+\left[ 1+\alpha _{n}\beta _{n}\delta \right] \varphi \left( \left\Vert
u_{n}-Tu_{n}\right\Vert \right) \text{.}  \label{eqn44}
\end{eqnarray}%
Since $\delta ,\alpha _{n},\beta _{n}\in \left[ 0,1\right) $ for all $n\in 
\mathbb{N}
$,%
\begin{equation}
\left( 1-\alpha _{n}\right) \delta <1-\alpha _{n}\text{, }1-\beta _{n}\left(
1-\delta \right) <1\text{.}  \label{eqn45}
\end{equation}%
Using (2.27) and the assumption $\alpha _{n}\geq A>0$,$\forall n\in 
\mathbb{N}
$ in (2.26) we derive%
\begin{eqnarray}
\left\Vert u_{n+1}-x_{n+1}\right\Vert &\leq &\left[ 1-A\left( 1-\delta
\right) \right] \left\Vert u_{n}-x_{n}\right\Vert  \notag \\
&&+\left[ 1-A\left( 1-\delta \right) \right] \left\Vert
u_{n}-Tu_{n}\right\Vert  \notag \\
&&+\left[ 1+\alpha _{n}\beta _{n}\delta \right] \varphi \left( \left\Vert
u_{n}-Tu_{n}\right\Vert \right) \text{.}  \label{eqn46}
\end{eqnarray}%
Define%
\begin{eqnarray*}
a_{n} &:&=\left\Vert u_{n}-x_{n}\right\Vert \text{,} \\
\mu _{n} &:&=A\left( 1-\delta \right) \in \left( 0,1\right) \text{,} \\
\rho _{n} &:&=\left[ 1-A\left( 1-\delta \right) \right] \left\Vert
u_{n}-Tu_{n}\right\Vert +\left[ 1+\alpha _{n}\beta _{n}\delta \right]
\varphi \left( \left\Vert u_{n}-Tu_{n}\right\Vert \right) \text{.}
\end{eqnarray*}%
Since $\lim_{n\rightarrow \infty }\left\Vert u_{n}-p\right\Vert =0$, $%
\lim_{n\rightarrow \infty }\left\Vert u_{n}-Tu_{n}\right\Vert =0$ as in the
proof of Theorem1. It therefore follows, using the same argument as that
employed in the proof of Theorem 1 that $\lim_{n\rightarrow \infty }x_{n}=p$.

We will prove now that, if the S-iteration converges, then the Mann
iteration does too.

Using (1.4), the Mann iteration \cite{Mann}, and (1.7) we have%
\begin{eqnarray}
\left\Vert x_{n+1}-u_{n+1}\right\Vert &=&\left\Vert \left( 1-\alpha
_{n}\right) \left( Tx_{n}-u_{n}\right) +\alpha _{n}\left(
Ty_{n}-Tu_{n}\right) \right\Vert  \notag \\
&\leq &\left( 1-\alpha _{n}\right) \left\Vert Tx_{n}-u_{n}\right\Vert
+\alpha _{n}\left\Vert Ty_{n}-Tu_{n}\right\Vert  \notag \\
&\leq &\left( 1-\alpha _{n}\right) \left\Vert Tx_{n}-u_{n}\right\Vert
+\alpha _{n}\delta \left\Vert y_{n}-u_{n}\right\Vert +\alpha _{n}\varphi
\left( \left\Vert y_{n}-Ty_{n}\right\Vert \right) \text{.}  \label{eqn47}
\end{eqnarray}%
We now have the following estimates%
\begin{eqnarray}
\left\Vert y_{n}-u_{n}\right\Vert &=&\left\Vert \left( 1-\beta _{n}\right)
x_{n}+\beta _{n}Tx_{n}-u_{n}\right\Vert  \notag \\
&=&\left\Vert \left( 1-\beta _{n}\right) x_{n}+\beta _{n}Tx_{n}-u_{n}-\beta
_{n}u_{n}+\beta _{n}u_{n}\right\Vert  \notag \\
&\leq &\left( 1-\beta _{n}\right) \left\Vert x_{n}-u_{n}\right\Vert +\beta
_{n}\left\Vert Tx_{n}-u_{n}\right\Vert \text{,}  \label{eqn48}
\end{eqnarray}%
\begin{eqnarray}
\left\Vert Tx_{n}-u_{n}\right\Vert &=&\left\Vert
Tx_{n}-x_{n}+x_{n}-u_{n}\right\Vert  \notag \\
&\leq &\left\Vert Tx_{n}-x_{n}\right\Vert +\left\Vert x_{n}-u_{n}\right\Vert 
\text{.}  \label{eqn49}
\end{eqnarray}%
Relations (2.29),(2.30), and (2.31) lead to%
\begin{eqnarray}
\left\Vert x_{n+1}-u_{n+1}\right\Vert &\leq &\left[ 1-\alpha _{n}\left(
1-\delta \right) \right] \left\Vert x_{n}-u_{n}\right\Vert  \notag \\
&&+\left[ 1-\alpha _{n}+\alpha _{n}\beta _{n}\delta \right] \left\Vert
Tx_{n}-x_{n}\right\Vert +\alpha _{n}\varphi \left( \left\Vert
y_{n}-Ty_{n}\right\Vert \right) \text{.}  \label{eqn50}
\end{eqnarray}%
Since $\beta _{n}\in \left[ 0,1\right) $ for all $n\in 
\mathbb{N}
$,%
\begin{equation}
\alpha _{n}\beta _{n}\delta <\alpha _{n}\delta \text{.}  \label{eqn51}
\end{equation}%
Utilizing inequality (2.33) and the assumption $\alpha _{n}\geq A>0$,$%
\forall n\in 
\mathbb{N}
$ in (2.32) we get%
\begin{eqnarray}
\left\Vert u_{n+1}-x_{n+1}\right\Vert &\leq &\left[ 1-A\left( 1-\delta
\right) \right] \left\Vert x_{n}-u_{n}\right\Vert  \notag \\
&&+\left[ 1-A\left( 1-\delta \right) \right] \left\Vert
Tx_{n}-x_{n}\right\Vert +\alpha _{n}\varphi \left( \left\Vert
y_{n}-Ty_{n}\right\Vert \right) \text{.}  \label{eqn52}
\end{eqnarray}%
Now define%
\begin{eqnarray*}
a_{n} &:&=\left\Vert x_{n}-u_{n}\right\Vert \text{,} \\
\mu _{n} &:&=A\left( 1-\delta \right) \in \left( 0,1\right) \text{,} \\
\rho _{n} &:&=\left[ 1-A\left( 1-\delta \right) \right] \left\Vert
Tx_{n}-x_{n}\right\Vert +\alpha _{n}\varphi \left( \left\Vert
y_{n}-Ty_{n}\right\Vert \right) \text{.}
\end{eqnarray*}%
Since $\lim_{n\rightarrow \infty }\left\Vert x_{n}-p\right\Vert =0$, $%
\lim_{n\rightarrow \infty }\left\Vert Tx_{n}-x_{n}\right\Vert =0$ as in the
proof of Theorem1.

Now we have%
\begin{eqnarray}
0 &\leq &\left\Vert y_{n}-Ty_{n}\right\Vert  \notag \\
&\leq &\left\Vert y_{n}-p\right\Vert +\left\Vert Tp-Ty_{n}\right\Vert  \notag
\\
&\leq &\left\Vert y_{n}-p\right\Vert +\delta \left\Vert p-y_{n}\right\Vert
+\varphi \left( \left\Vert p-Tp\right\Vert \right)  \notag \\
&=&\left( 1+\delta \right) \left\Vert y_{n}-p\right\Vert  \notag \\
&\leq &\left( 1+\delta \right) \left( 1-\beta _{n}\right) \left\Vert
x_{n}-p\right\Vert +\left( 1+\delta \right) \beta _{n}\left\Vert
Tx_{n}-Tp\right\Vert  \notag \\
&\leq &\left( 1+\delta \right) \left[ 1-\beta _{n}+\beta _{n}\right]
\left\Vert x_{n}-p\right\Vert +\left( 1+\delta \right) \beta _{n}\varphi
\left( \left\Vert p-Tp\right\Vert \right)  \notag \\
&=&\left( 1+\delta \right) \left\Vert x_{n}-p\right\Vert \rightarrow 0\text{
as }n\rightarrow \infty \text{,}  \label{eqn53}
\end{eqnarray}%
that is, $\lim_{n\rightarrow \infty }\left\Vert y_{n}-Ty_{n}\right\Vert =0$,
threfore using the same argument as in the proof of Theorem 1, it can be
shown that $\lim_{n\rightarrow \infty }u_{n}=p$.
\end{proof}

As shown by \c{S}oltuz and Grosan (\cite{Data Is 2}, Theorem 3.1), in a real
Banach space $X$, the Ishikawa iteration \cite{Ishikawa} converges to the
fixed point of $T$, where $T:E\rightarrow E$ is a mapping satisfying
condition (1.4).

In 2007, \c{S}oltuz (\cite{Mult Soltz}, Corollary 2) proved that the
Krasnoselskij \cite{Krasnoselskij}, Mann \cite{Mann}, Ishikawa \cite%
{Ishikawa}, Noor \cite{Noor} and multistep (1.5) iterations are equivalent
for quasi-contractive mappings in a normed space setting.

In 2011, Chugh and Kumar (\cite{CR}, Corollary 3.2) proved that the Picard 
\cite{Picard}, Mann \cite{Mann}, Ishikawa \cite{Ishikawa}, new two step \cite%
{Thianwan}, Noor \cite{Noor} and SP \cite{SP} iterations are equivalent for
quasi-contractive mappings in a Banach space setting.

From the argument used in the proofs of (\cite{Data Is 2}, Theorem 3.1), (%
\cite{Mult Soltz}, Corollary 2) and (\cite{CR}, Corollary 3.2) we easily
obtain the following corollary:

\begin{corollary}
$T:E\rightarrow E$ be a mapping satisfying condition $(1.4)$ with $F_{T}\neq
\emptyset $. If the initial point is the same for all iterations, $\alpha
_{n}\geq A>0$, $\forall n\in 
\mathbb{N}
$, then the \textit{following are equivalent:}
\end{corollary}

\begin{enumerate}
\item \textit{The Picard iteration }\cite{Picard}\textit{\ converges to }$%
p\in F_{T}$\textit{;}

\item \textit{The Krasnoselskij iteration \cite{Krasnoselskij}\ converges to 
}$p\in F_{T}$\textit{;}

\item \textit{The Mann iteration \cite{Mann}\ converges to }$p\in F_{T}$%
\textit{;}

\item \textit{The Ishikawa iteration }\cite{Ishikawa}\textit{\ converges to }%
$p\in F_{T}$\textit{;}

\item \textit{The new two step iteration }\cite{Thianwan}\textit{\ converges
to }$p\in F_{T}$\textit{;}

\item \textit{The Noor iteration }\cite{Noor}\textit{\ converges to }$p\in
F_{T}$\textit{;}

\item \textit{The SP iteration }\cite{SP}\textit{\ converges to }$p\in F_{T}$%
\textit{;}

\item \textit{The Multistep iteration }(1.5)\textit{\ converges to }$p\in
F_{T}$\textit{;}
\end{enumerate}

Together with Theorem 1 and Theorem 2,Corollary 1 leads to the following
corollary:

\begin{corollary}
$T:E\rightarrow E$ be a mapping satisfying condition $(1.4)$ with $F_{T}\neq
\emptyset $. If the initial point is the same for all iterations, $\alpha
_{n}\geq A>0$, $\forall n\in 
\mathbb{N}
$, then the \textit{following are equivalent:}
\end{corollary}

\begin{enumerate}
\item \textit{The Picard iteration }\cite{Picard}\textit{\ converges to }$%
p\in F_{T}$\textit{;}

\item \textit{The Krasnoselskij iteration \cite{Krasnoselskij}\ converges to 
}$p\in F_{T}$\textit{;}

\item \textit{The Mann iteration \cite{Mann}\ converges to }$p\in F_{T}$%
\textit{;}

\item \textit{The Ishikawa iteration }\cite{Ishikawa}\textit{\ converges to }%
$p\in F_{T}$\textit{;}

\item \textit{The new two step iteration }\cite{Thianwan}\textit{\ converges
to }$p\in F_{T}$\textit{;}

\item \textit{The Noor iteration }\cite{Noor}\textit{\ converges to }$p\in
F_{T}$\textit{;}

\item \textit{The SP iteration }\cite{SP}\textit{\ converges to }$p\in F_{T}$%
\textit{;}

\item \textit{The Multistep iteration }(1.5)\textit{\ converges to }$p\in
F_{T}$\textit{;}

\item \textit{The new multistep iteration} (1.6)\textit{\ converges to }$%
p\in F_{T}$\textit{;}

\item \textit{The S-iteration }(1.7)\textit{\ converges to }$p\in F_{T}$.
\end{enumerate}

\begin{acknowledgement}
The first two authors would like to thank Y\i ld\i z Technical University
Scientific Research Projects Coordination Department under project number
BAPK 2012-07-03-DOP02 for financial support during the preparation of this
manuscript.
\end{acknowledgement}


\begin{thebibliography}{99}
\bibitem{RS1} B.E. Rhoades, S.M. \c{S}oltuz, The equivalence of Mann
iteration and Ishikawa iteration for non-Lipschitzian operators, Int. J.
Math. Math. Sci. 42(2003) 2645--2651

\bibitem{RS2} B.E. Rhoades, S.M. \c{S}oltuz, The equivalence between the
convergences of Ishikawa and Mann iterations for an asymptotically
pseudocontractive map, J. Math. Anal. Appl. 283(2003) 681--688.

\bibitem{RS3} B.E. Rhoades, S.M. \c{S}oltuz, On the equivalence of Mann and
Ishikawa iteration methods, Int. J. Math. Math. Sci. 7(2001) 451--459 .

\bibitem{RS4} B.E. Rhoades, S.M. \c{S}oltuz, The equivalence of Mann
iteration and Ishikawa iteration for a Lipschitzian $\Psi $-uniformly
pseudocontractive and $\Psi $-uniformly accretive maps, Tamkang J. Math.
35(2004) 235--245.

\bibitem{RS5} B.E. Rhoades, S.M. \c{S}oltuz, The equivalence between the
convergences of Ishikawa and Mann iterations for an asymptotically
nonexpansive in the intermediate sense and strongly successively
pseudocontractive maps, J.Math. Anal. Appl. 289(2004) 266--278.

\bibitem{RS6} B.E. Rhoades, S.M. \c{S}oltuz, The equivalence of Mann and
Ishikawa iteration for $\Psi $-uniformly pseudocontractive or $\Psi $%
-uniformly accretive maps, Int. J. Math. Math. Sci. 46(2004) 2443--2452.

\bibitem{RS7} B.E. Rhoades, S.M. \c{S}oltuz, The equivalence between
Mann-Ishikawa iterations and multistep iteration, Nonlinear Analysis
58(2004) 219-228.

\bibitem{RS8} B.E. Rhoades, S.M. \c{S}oltuz, The equivalence of Mann and
Ishikawa iteration dealing with strongly pseudocontractive or strongly
accretive maps, Panamer. Math. J. 14(2004) 51--59.

\bibitem{Chidume} C.E. Chidume, S.A. Mutangadura, An example on the Mann
iteration method for Lipschitz pseudocontractions, Proc. Am. Math. Soc. 129
(2001) 2359--2363.

\bibitem{Imoru} C.O. Imoru, M.O. Olantiwo, On the stability of Picard and
Mann iteration processes, Carpathian Journal of Mathematics 19(2003) 155-160.

\bibitem{Picard} E. Picard, Memoire sur la theorie des equations aux
derivees partielles et la methode des approximations successives, J. Math.
Pures et Appl. \textbf{6}(1890), 145-210.

\bibitem{New} F. G\"{u}rsoy, V. Karakaya, B.E. Rhoades, Data dependence
results of a new multistep and S-iterative schemes for contractive-like
operators, Fixed Point Theory and Applications 2013, \textbf{2013}:76.
doi:10.1186/1687-1812-2013-76.

\bibitem{Akewe} H. Akewe, Strong convergence and stabilitiy of
Jungck-multistep-SP iteration for generalized contractive-like inequality
operators, Advances in Natural Science 5(2012) 21-27.

\bibitem{Olaleru} J.O. Olaleru, H. Akewe, The equivalence of Jungck-type
iterations for generalized contractive-like operators in a Banach space,
Fasciculi Mathematici (2011) 47-61.

\bibitem{Krasnoselskij} M.A. Krasnoselkij, Two remarks on the method of
successive approximations, Uspehi Mat. Nauk. 63(1)(1955) 123-127.

\bibitem{Noor} M.A. Noor, New approximation schemes for general variational
inequalities, Journal of Mathematical Analysis and Applications 251(2000)
217-229.

\bibitem{Osilike} M.O. Osilike, A. Udomene, Short proofs of stability
results for fixed point iteration procedures for a class of contractive-type
mappings, Indian Journal of Pure and Applied Mathematics 30(1999) 1229-1234.

\bibitem{SP} Phuengrattana, Withunand, Suantai, Suthep, On the rate of
convergence of Mann, Ishikawa, Noor and SP iterations for continuous
functions on an arbitrary interval, Journal of Computational and Applied
Mathematics 235(2011) 3006-3014.

\bibitem{Agarwal} R.P. Agarwal, D. O'Regan, D.R. Sahu, Fixed point theory
for lipschitzian type-mappings with applications, Springer 2009.

\bibitem{Agarwall} R.P. Agarwal, D. O'Regan, D.R. Sahu, Iterative
construction of fixed points of nearly asymptotically nonexpansive mappings,
J. Nonlinear Convex Anal. 8(2007) 61-79.

\bibitem{CR} R. Chugh, V. Kumar, Strong convergence of SP iterative scheme
for quasi-contractive operators in Banach spaces, International Journal of
Computer Applications 31(2011) 21-27.

\bibitem{Glowinski} R. Glowinski, P. Le Tallec, Augmented Lagrangian and
operatorsplitting methods in nonlinear mechanics, SIAM, Philadelphia,1989.

\bibitem{Chang} S.S. Chang, Y.J. Cho, J.K. Kim, The equivalence between the
convergence of modified Picard, modified Mann, and modified Ishikawa
iterations, Math. Comput. Modelling 37(2003) 985--991.

\bibitem{Ishikawa} S. Ishikawa, Fixed points by a new iteration method,
Proc. Amer. Math. Soc. 44(1974) 147-150.

\bibitem{S1} S.M. \c{S}oltuz, An equivalence between the convergences of
Ishikawa, Mann and Picard iterations, Math. Commun. 8(2003) 15--22.

\bibitem{S2} S.M. \c{S}oltuz, A remark concerning the paper: An equivalence
between the convergences of Ishikawa, Mann and Picard iterations, Rev. Anal.
Numer. Theor. Approx. 33(2004) 95--96.

\bibitem{Data Is 2} S.M. \c{S}oltuz, T. Grosan, Data dependence for Ishikawa
iteration when dealing with contractive like operators, Fixed Point Theory
and Applications 2008(2008) Article ID 242916 7 pages.

\bibitem{Mult Soltz} S.M. \c{S}oltuz, The equivalence between Krasnoselskij,
Mann, Ishikawa, Noor and multistep iterations, Mathematical Communications
12(2007) 53-61.

\bibitem{Thianwan} S. Thianwan, Common fixed points of new iterations for
two asymptotically nonexpansive nonself mappings in a Banach space, J.
Comput. Appl. Math. (2008) doi: 10.1016/j.cam.2008.05.051.

\bibitem{Zamfirescu} T. Zamfirescu, Fix point theorems in metric spaces,
Archiv der Mathematik 23(1972) 292-298.

\bibitem{Berinde1} V. Berinde, Picard iteration converges faster than the
Mann iteration in the class of quasi-contractive operators, Fixed Point
Theory Appl. (2004) 97--105.

\bibitem{Berinde} V. Berinde, On the convergence of the Ishikawa iteration
in the class of quasi contractive operators, Acta Mathematica Universitatis
Comenianae 73(2004) 119-126.

\bibitem{Berinde2} V. Berinde, Iterative approximation of fixed points,
Springer Berlin Heidelberg, New York, 2007.

\bibitem{Mann} W.R. Mann, Mean value methods in iterations, Proc. Amer.
Math. Soc. 4(1953) 506-510.

\bibitem{Weng} X. Weng, Fixed point iteration for local strictly
pseudocontractive mapping, Proc. Amer. Math. Soc. 113(1991) 727-731.

\bibitem{Yuguang} Y. Xu, Ishikawa and Mann Iterative Processes with Errors
for Nonlinear Strongly Accretive Operator Equations, Journal of Mathematical
Analysis and Applications 224(1998) 91-101.

\bibitem{Zhenyu} Z. Huang, Mann and Ishikawa iterations with errors for
asymptotically nonexpansive mappings, Computers \& Mathematics with
Applications 37(1999) 1-7.
\end{thebibliography}
\end{document}